\documentclass{daj}

\usepackage{geometry}
\usepackage{amsmath}
\usepackage{amssymb}
\usepackage{amsthm}
\usepackage{cleveref}
\usepackage{mathtools}
\usepackage{tikz}
\usepackage{comment}
\usetikzlibrary{decorations.markings}
\theoremstyle{definition}
\newtheorem{thm}{Theorem}[section]
\newtheorem{conj}[thm]{Conjecture}
\newtheorem{rmk}[thm]{Remark}
\newtheorem{prop}[thm]{Proposition}
\newtheorem{clm}[thm]{Claim}
\newtheorem{defn}[thm]{Definition}
\newtheorem{exmp}[thm]{Example}
\newtheorem{lem}[thm]{Lemma}

\newcommand{\co}{\operatorname{co}}

\dajAUTHORdetails{%
  title = {Sharp $L^1$ Inequalities for Sup-Convolution}, %% please capitalize all significant words
  author = {Peter van Hintum, Hunter Spink, and Marius Tiba},
  plaintextauthor = {Peter van Hintum, Hunter Spink, Marius Tiba},
  plaintexttitle = {Sharp L1 Inequalities for Sup-Convolution}, %%  title without math or LaTeX
}   %%% END \dajAUTHORdetails

\dajEDITORdetails{%
   year={2023},
   number={7},
   received={12 August 2020},   % received date: example: 7 January 2017
   published={4 July 2023},  % published date
   doi={10.19086/da.81314},       % XXX = number of paper, e.g. da006 for paper#6
}   %%% END \dajEDITORdetails

\begin{document}

\begin{frontmatter}[classification=text]
%% EDITOR: this will force the keywords to appear right after the Abstract.
%%   If the abstract is too long and would force the keywords off the
%%   front page, please comment out % [classification=text] above
%%   This way the keywords will be floated on the bottom of the first page
%%   even though the Abstract spills over to the next page.

%%% AUTHOR: Title goes here.  This line is optional.  You must use it
%%   if title has footnote attached or requires nontrivial typesetting,
%%   e.g., inclusion of linebreaks to force nice layout.
\title{Sharp $L^1$ Inequalities for Sup-Convolution} %% please capitalize all significant words

%%% AUTHOR:
%%% List all authors. If you wish, place grant acknowledgements in \thanks.
%%% In brackets include a short tag for each author.
\author[pvh]{Peter van Hintum}
\author[hs]{Hunter Spink}
\author[mt]{Marius Tiba}

%%% AUTHOR: Abstract goes here
\begin{abstract}
Given a compact convex domain $C\subset \mathbb{R}^k$ and bounded measurable functions $f_1,\ldots,f_n:C\to \mathbb{R}$, define the sup-convolution $(f_1\ast \ldots \ast f_n)(z)$ to be the supremum average value of $f_1(x_1),\ldots,f_n(x_n)$ over all $x_1,\ldots,x_n\in C$ which average to $z$. Continuing the study by Figalli and Jerison and the present authors of linear stability for the Brunn-Minkowski inequality with equal sets, for $k\le 3$ we find the optimal constants $c_{k,n}$ such that
$$\int_C f^{\ast n}(x)-f(x) dx \ge c_{k,n}\int_C\co(f)(x)-f(x) dx$$ where $\co(f)$ is the upper convex hull of $f$. Also, we show $c_{k,n}=1-O(\frac{1}{n})$ for fixed $k$ and prove an analogous optimal inequality for two distinct functions. The key geometric insight is a decomposition of polytopal approximations of $C$ into hypersimplices according to the geometry of the set of points where $\co(f)$ is close to $f$.
\end{abstract}
\end{frontmatter}

%%% AUTHOR: body of paper starts here

\section{Introduction}
Let $C\subset \mathbb{R}^k$ be a compact convex domain. For a bounded function $f:C\to \mathbb{R}$, $\co(f)$ is defined to be the upper convex hull of $f$ (the infimum of all concave functions larger than $f$), and for bounded measurable functions $f_1,\ldots,f_n:C\to \mathbb{R}$, the \emph{sup-convolution} is defined to be
$$(f_1 \ast \ldots \ast f_n)(z):=\sup\left\{\frac{f_1(x_1)+\ldots+f(x_n)}{n} : \frac{x_1+\ldots +x_n}{n}=z\right\}.$$
The operation of sup-convolution, or in its equivalent form inf-convolution $-((-f_1)\ast \ldots \ast (-f_n))$, naturally appears in problems of optimization, with $f_1,\ldots,f_n$ utility functions and $C$ representing a cost domain \cite{Ekeland}. For a general survey, see \cite{Infconvsurvey}. Clearly $f_1\ast \ldots \ast f_n \ge \frac{f_1+\ldots+f_n}{n}$, and equality is attained when for example $f_1,\ldots,f_n$ are scalings of the same concave function $f=\co(f)$.

We can view the sup-convolution operation geometrically in terms of the Minkowski sum of regions in $\mathbb{R}^{k+1}$. Indeed, consider the hypograph $$A_{f,\lambda}=\{(x,y)\in C\times \mathbb{R}: \lambda \le y \le f(x)\}.$$
Then we have the closed convex hull $\overline{\co(A_{f,\lambda})}=A_{\co(f),\lambda}$, and for $\lambda$ sufficiently negative we have $$A_{f_1\ast \ldots \ast f_n,\lambda}=\frac{1}{n}(A_{f_1,\lambda}+\ldots+A_{f_n,\lambda}).$$

The study of how close a Minkowski sum is to its convex hull was started by Starr-Shapley-Folkman \cite{starr1969quasi} and Emerson-Greenleaf \cite{emerson1969asymptotic}, who showed that if $A_1,\ldots,A_n$ are subsets of the unit ball in $\mathbb{R}^k$, then the Hausdorff distance between the Minkowski averages $\frac{1}{n}(A_1+\ldots+A_n)$ and $\frac{1}{n}(\co(A_1)+\ldots+\co(A_n))$ is bounded above by $\sqrt{k}n^{-1}$. Of particular interest for us will be when $A_1=\ldots=A_n=A$, where we are concerned with how close $\frac{1}{n}(A+\ldots+A)$ is to $\co(A)$; for this equal sets case we refer the reader to the extensive survey \cite{fradelizi2018convexification}.

Ruzsa \cite[Theorem 5]{ruzsa1997brunn} showed that there is a constant $D_k$ such that for $A\subset \mathbb{R}^k$ of positive measure (taking the outer Lebesgue measure everywhere) and $n>D_k\frac{|\co(A)|}{|A|}$, we have $|\frac{1}{n}(A+\ldots+A)|\ge \left(1-\frac{D_k}{n}\cdot\frac{|\co(A)|}{|A|}\right)^k|\co(A)|$. In another direction, resolving a conjecture of Figalli and Jerison \cite{Semisum,FigJerJems} on the stability of the Brunn-Minkowski inequality for homothetic sets, the present authors \cite{HomoBM} showed that for $t\in (0,1)$ there are constants $c_k(t)$ and $d_k(t)$ such that for subsets $A\subset \mathbb{R}^k$ of positive measure, $|tA+(1-t)A| \ge c_{k}(t)|\co(A)\setminus A|$ provided $|(tA+(1-t)A)\setminus A|\le d_{k}(t)|A|$.

A nice feature of this last result is that for $A=A_{f,\lambda}$ the hypograph of a function, the $d_{k+1}(t)$ condition is always satisfied provided we take $\lambda$ to be sufficiently negative. Taking $t=\frac{1}{n}$ allows us to conclude, writing $f^{\ast n}$ for $f\ast \ldots \ast f$, that there exist positive constants $c_{k,n}$ such that\footnote{Formally we work with the ``upper Lebesgue integral'' to avoid the issue of the measurability of $f_1\ast\ldots \ast f_n$.} 
$$\int_C f^{\ast n}(x)-f(x)dx \ge c_{k,n}\int_C \co(f)(x)-f(x) dx$$
(see Appendix A, where we also give an alternate self-contained proof of this particular inequality).

\begin{comment}
\begin{thm}[vH,S,T]
For any $A\subset \mathbb{R}^{k+1}$ of positive measure and $t\in (0,1)$, there are constants $c(k,t)$ and $d(k,t)$ such that if $|tA+(1-t)A|-|A|\le d(k,t)|A|$, then
$$|tA+(1-t)A|-|A|\ge c(k,t)|\co(A)\setminus A|,$$
where we write $\co(A)$ for the convex hull of $A$.
\end{thm}

The existence of the constant $c_{k,n}$ for sup-convolution then follows by applying this theorem to the set $A=A_{f,-N}$ where $$A_{f,\lambda}=\{(x,y)\in C\times \mathbb{R}: \lambda \le y \le f(x)\},$$
$t=\frac{1}{n}$, and $N\le \min(f)$ is sufficiently small so that the $d(k,\frac{1}{n})$ bound is satisfied (in fact this shows that we have a lower bound even if we restricted in the $n$-fold sup-convolution that $x_1=\ldots=x_{n-1}$).
\end{comment}

 The constants $c_{k,n}$ one obtains in this way however are not optimal. Our first theorem establishes the optimal constants for $k\le 3$, making progress towards Question 1.8 from \cite{copro} which asked an analogous question in the discrete setting with $n=2$.

\begin{thm}\label{first_main}
If $f:C\to \mathbb{R}$ is a bounded measurable function with $C\subset \mathbb{R}^k$ a compact convex domain with $k\le 3$, and $n\ge 1$, then
$$\int_C f^{\ast n}(x)-f(x)dx \ge c_{k,n}\int_C \co(f)(x)-f(x) dx$$
with
$$c_{k,n}=\begin{cases}\frac{n-1}{n}&k=1\\\frac{(2n-1)(n-1)}{2n^2}&k=2\\\frac{(n-1)^2}{n^2}&k=3.\end{cases}$$
\begin{comment}
\begin{align*}
    c_{1,m}&= \frac{m-1}{m}\\
    c_{2,m}&=\frac{m-1}{m}\cdot \frac{1}{m^2}\cdot\binom{m+1}{2}+\frac{m-2}{m}\cdot \frac{1}{m^2}\cdot \binom{m}{2}=\frac{(2m-1)(m-1)}{2m^2}\\
    c_{3,m}&=\frac{m-1}{m}\cdot\frac{1}{m^3}\cdot\binom{m+2}{3}+\frac{m-2}{m}\cdot \frac{4}{m^3}\binom{m+1}{3}+\frac{m-3}{m}\cdot \frac{1}{m^3}\binom{m}{3}=\frac{(m-1)^2}{m^2}.
\end{align*}
\end{comment}
\end{thm}
This is sharp, taking $f$ the indicator function on the vertices of $C=T$ a simplex. In any dimension $k$, letting $e_1,\ldots,e_{k+1}$ be the standard basis vectors in $\mathbb{R}^{k+1}$ and identifying $C=T$ with the convex hull of $ne_1,\ldots,ne_{k+1}$ we will see that the level sets of this particular $f^{\ast n}$ induce a subdivision of $C$ into hypersimplices, where translates of the $m$'th $k$-dimensional hypersimplex $P_{k,m}:=[0,1]^{k+1}\cap \{\sum x_i=m\}$ appear $\binom{n+k-m}{k}$ times (see \Cref{HypersimplexSection}).

For example for $k=2$ (depicting the case $n=4$ below), $f^{\ast n}$ takes value $\frac{n-1}{n}$ in the shaded region, the union of $\binom{n+1}{2}$ translates of the triangle $P_{2,1}$ \begin{tikzpicture}[scale=0.25]
\filldraw[color=gray] (0,0)--(1,1)--(2,0)--cycle;
\end{tikzpicture}, and $\frac{n-2}{n}$ in the unshaded region, the union of $\binom{n}{2}$ translates of the triangle $P_{2,2}$ 
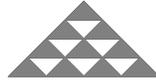
\begin{tikzpicture}[scale=0.25]
\draw[color=gray] (0,1)--(1,0)--(2,1)--cycle;
\end{tikzpicture}:
\begin{center}
\begin{tikzpicture}[scale=0.25]
\draw (0,0)--(4,4)--(8,0)--(0,0);

\draw (1,1)--(7,1);
\draw (2,2)--(6,2);
\draw (3,3)--(5,3);

\draw (6,0)--(7,1);
\draw (4,0)--(6,2);
\draw (2,0)--(5,3);

\draw (6,0)--(3,3);
\draw (4,0)--(2,2);
\draw (2,0)--(1,1);

\filldraw[color=gray] (0,0)--(1,1)--(2,0)--cycle;
\filldraw[color=gray] (2,0)--(3,1)--(4,0)--cycle;
\filldraw[color=gray] (4,0)--(5,1)--(6,0)--cycle;
\filldraw[color=gray] (6,0)--(7,1)--(8,0)--cycle;

\filldraw[color=gray] (1,1)--(2,2)--(3,1)--cycle;
\filldraw[color=gray] (3,1)--(4,2)--(5,1)--cycle;
\filldraw[color=gray] (5,1)--(6,2)--(7,1)--cycle;

\filldraw[color=gray] (2,2)--(3,3)--(4,2)--cycle;
\filldraw[color=gray] (4,2)--(5,3)--(6,2)--cycle;

\filldraw[color=gray] (3,3)--(4,4)--(5,3)--cycle;
\end{tikzpicture}
\end{center}
The shaded regions are precisely those parts of $T$ whose points can be expressed as $\frac{x_1+\ldots+x_n}{n}$ with all but one of the $x_i$ a vertex of $T$, and the remaining unshaded regions can be expressed with all but two of the $x_i$ a vertex of $T$.

For $k=3$, we can subdivide $T$ into $\binom{n+2}{3}$ translates of $\frac{1}{n}T=P_{3,1}$, $\binom{n+1}{3}$ translates of the octahedron $P_{3,2}$, and $\binom{n}{3}$ translates of $-P_{3,1}=P_{3,3}$, on which $f^{\ast n}$ takes the values $\frac{n-1}{n}$, $\frac{n-2}{n}$, $\frac{n-3}{n}$ respectively. The partition is according to whether the maximum number of vertices of $T$ which can be used to express the point as an $n$-average is $n-1,n-2,$ or $n-3$ respectively.

To prove \Cref{first_main}, we pass to a piecewise-linear approximation and then triangulate according to the domains of linearity of $\co(f)$. On each simplex $T$ we prove a sharp inequality relating $\int_R \co(f)(x)-f^{\ast n}(x) dx$ and $\int_T \co(f)(x)-f(x) dx$ for $R$ ranging over the hypersimplices in the subdivision alluded to above. This in turn is encompassed in our notion of an ``$m$-averageable'' subset of $T$ (Section 4), and showing certain hypersimplices are ``$m$-averageable'' allows us to conclude.

We make the following conjecture for arbitrary $k,n \ge 1$. In what follows, write $A(k,\ell)$ for the Eulerian number counting permutations of $S_{k}$ with $\ell$ descents.
\begin{comment}
These decompositions generalize for all $k$ to a decomposition of $\mathbb{R}^k$ into hypersimplices, obtained by identifying $\mathbb{R}^k$ with some hyperplane $\sum x_i=t$ with $t\in \mathbb{Z}$ in $\mathbb{R}^{k+1}$, and partitioning with all hyperplanes $x_i=\ell$ with $i\in \{1,\ldots,k+1\}$ and $\ell\in \mathbb{Z}$. Taking $t=n$, the decomposition associated to $f^{\ast n}$ is naturally identified with the induced decomposition of the $k$-dimensional simplex $T$ with vertices $\{ne_i\}_{i=1}^{k+1}$, where $e_i$ are standard basis vectors in $\mathbb{R}^{k+1}$.

Indeed, the regions of $\mathbb{R}^k$ cut out by these hyperplanes are obtained as the intersection of a unit hypercube $[0,1]^{k+1}+x$ with $x\in \mathbb{Z}^k$ with the hyperplane $\sum x_i=t$, which if not empty or a single point is of the form $P_{k,m}+x$ for some $m\in \{1,\ldots,k\}$ where $$P_{k,m}:=[0,1]^{k+1} \cap \left\{\sum x_i=m\right\}$$ is the $m$'th $k$-dimensional hypersimplex. Taking $t=n$, for fixed $m$ the union of the regions of the form $P_{k,m}+x$ lying in $T$ are then easily shown to be those points for which the maximum number of vertices of $T$ which can be used to express them as an $n$-average is $n-m$.
\end{comment}

\begin{conj}\label{constantconj}
If $k, n\ge 1$ and $f:C\to \mathbb{R}$ is a bounded measurable function with $C\subset \mathbb{R}^k$ a compact convex domain, then we have
$$\int_C f^{\ast n}(x)-f(x)dx \ge c_{k,n}\int_C \co(f)(x)-f(x)dx,$$
where
$$c_{k,n}=\frac{1}{n^k}\sum_{m=1}^k\frac{n-m}{n}\binom{n+k-m}{k}A(k,m-1)=\frac{k+1}{n^{k+1}}(1^k+\ldots+(n-1)^k).$$
\end{conj}
If true, this would be sharp by taking $f$ the indicator function on the vertices of $C=T$ a simplex (see \Cref{HypersimplexSection}). Here $\frac{n-m}{m}$ is the value of $f^{\ast n}$ on each hypersimplex $P_{k,m}+x$, $\binom{n+k-m}{k}$ is the number of such hypersimplices, and $\frac{A(k,m-1)}{n^k}$ is the volume ratio of $P_{k,m}$ to $T$.

\begin{rmk}
Omitting the $\frac{n-m}{m}$ factor, we obtain a geometric proof of the Worpitzky identity $n^k=\sum \binom{n+k-m}{k}A(k,m-1).$ A similar observation was recently exploited by Early \cite[Section 3]{1611.06640} to categorify the Worpitzky identity via the representation theory of the symmetric group.
\end{rmk}

We also show the following asymptotic result for fixed $k$ as $n\to \infty$.
\begin{thm}
\label{AsympTheorem}
For any $k\ge 1$ and $n \ge k+1$, we have $c_{k,n}\ge 1-\binom{n}{k}\frac{k^{k+1}}{n^{k+1}}=1-O(\frac{1}{n})$.
\end{thm}
This is optimal up to the constant on $\frac{1}{n}$, which this theorem shows can be taken to be $\frac{k^{k+1}}{k!}=e^{O(k)}$, though our conjectured extremal example gives a constant of $\frac{k+1}{2}$.

Finally, we consider the sup-convolution of distinct functions $f,g$, showing that $f$ is close to $\co(f)$ provided $f\ast g$ is close to $\frac{f+g}{2}$.
\begin{thm}\label{second_main}
If $f,g:C\to \mathbb{R}$ are bounded measurable functions with $C\subset \mathbb{R}^k$ a compact convex domain and $k\le 3$ then
$$\int_C f*g(x)-\frac{f(x)+ g(x)}{2}dx\geq \frac{k+1}{2^{k+1}}\int_C \co(f)(x)-f(x)dx.$$
\end{thm}
The constant $c_{k,2}=\frac{k+1}{2^{k+1}}$ is again sharp, as for example we can take $f=g$ the indicator function on the vertices of $C=T$ a simplex.

In Section 2 we show that  \Cref{first_main}, \Cref{constantconj}, \Cref{AsympTheorem}, and \Cref{second_main} reduce to the case that $C=T$ is a simplex, $f=0$ on the vertices and $f\le 0$ otherwise. In Section 3 we construct our hypersimplex subdivision of $T$. In Section 4 we introduce a new geometric notion of ``$m$-averageable subsets of $T$'', and reduce to showing certain hypersimplices in $T$ are ``$m$-averageable''. In Section 5 we show that the relevant hypersimplices up to dimension $3$ are ``$m$-averageable'' and conclude \Cref{first_main} and \Cref{second_main}. In Section 6 we prove \Cref{AsympTheorem}. Finally, in Appendix A we show how the existence of a non-sharp constant in \Cref{constantconj} can be derived from \cite{HomoBM}, and we also give a quick self-contained proof.

\section{Reduction to Simplices}\label{reduction}

Here we reduce \Cref{first_main}, \Cref{constantconj}, \Cref{AsympTheorem}, and \Cref{second_main} to the case $C=T$ is a simplex, $f\le 0$, and $f=0$ at the vertices.

\begin{prop}\label{mainreduction}
The statements of \Cref{first_main}, \Cref{constantconj}, \Cref{AsympTheorem}, and \Cref{second_main}, respectively, are equivalent to the corresponding statements with the additional assumption that $C=T$ is a simplex, $f\le 0$, and $f=0$ at the vertices.
\end{prop}

\begin{proof}

The reduction is divided in the following three steps. The first two steps will reduce to the situation that $f$ is nonnegative with piecewise-linear $\co(f)$ with finitely many domains of linearity. Considering a particular domain of linearity $T$, by subtracting the linear function $\co(f)|_T$, we deduce the result. We shall always focus on the reduction of \Cref{constantconj}, as the others follow in a similar way. 

\begin{clm}\label{firsthunterprop}
Suppose that \Cref{first_main} (resp. \Cref{constantconj}, \Cref{AsympTheorem},  \Cref{second_main}) is true when the domain $C$ is a polytope $P$, $f \ge 0$, and $f=0$ on a neighborhood of $\partial C$. Then \Cref{first_main} (resp. \Cref{constantconj}, \Cref{AsympTheorem}, \Cref{second_main}) is true.
\end{clm}
\begin{proof}
We prove this claim for \Cref{constantconj}, the other cases are similar. The inequality doesn't change if we scale $f$ or add a constant so assume that $f(x)\in [n,n+1]$ for all $x\in C$. Let $P_1,P_2,\ldots$ be a sequence of polytopes with $C\subset P_i^{\circ}$ (the interior of $P_i$) and $|P_i|\to |C|$. We extend $f$ to a function $f_i$ on $P_i$ by setting $f_i=0$ on $P_i\setminus C$. Then we note that for any $x\in C$, $f_i^{\ast n}(x)\ge f_i(x)\ge n$, but
$$\frac{f_i(x_1)+\ldots+f_i(x_n)}{n}\le \frac{(n+1)(n-1)}{n}<n$$ provided any $x_j\in \partial C$, so we conclude that $f_i^{\ast n}|_C=f^{\ast n}$. 

%%%We further claim that $\int_{P_i}\co(f_i) dx\to \int_C \co(f) dx$. Indeed, suppose that $0\in C$ and note that by convexity of the hypograph of $f$, for any $\epsilon$ we have $(1+\epsilon)f(x/(1+\epsilon))\ge f_\ell$ as soon as $P_\ell \subset (1+\epsilon)C$.

Thus as $\co(f_i)\ge \co(f)$,
\begin{align*}\int_C f^{\ast n}(x)-f(x) dx &\ge \int_{P_i} f_i^{\ast n}(x)-f_i(x) dx - |P_i \setminus C|\cdot ||f_i^{\ast n}||_{\infty} \\
&\ge c_{k,n}\int_{P_i} \co(f_i)(x)-f_i(x) dx - |P_i\setminus C|\cdot ||f_i^{\ast n}||_{\infty}\\
&\ge c_{k,n}\int_C\co(f)(x)-f(x) dx-|P_i\setminus C|\cdot ||f_i^{\ast n}||_{\infty}\to c_{k,n}\int_C \co(f)(x)-f(x) dx,\end{align*}
where in the last step we used that $||f_i^{\ast n}||_{\infty} = ||f||_{\infty} \le n+1$.
\end{proof}
\begin{clm}\label{secondhunterprop}
Suppose that \Cref{first_main} (resp. \Cref{constantconj}, \Cref{second_main}) is true when the domain $C$ is a polytope $P$ and $\co(f)$ is the upper convex hull of finitely many points $(r_i,f(r_i))$ (so is in particular piecewise linear). Then it is true when the domain $C$ is a polytope $P$, $f \ge 0$, and $f=0$ on a neighborhood of $\partial C$.
\end{clm}
\begin{proof}
We prove this claim for \Cref{constantconj}, the other cases are similar. Suppose $C=P$, $f\ge 0$, and $f=0$ on a neighborhood of $\partial C$ (but we do not necessarily know that $\co(f)$ is piecewise linear).

We'll show that $\co(f)$ is continuous at all points $x\in C$. First, suppose that $x\in C^{\circ}$. Then for $y\in C^{\circ}$, let $z_1,z_2$ be the points on $\partial C$ such that $z_1,x,y,z_2$ are collinear in that order. We have
\begin{align*}\co(f)(x) &\ge \frac{||x-z_1||}{||y-z_1||}\co(f)(y)+\frac{||x-y||}{||y-z_1||} \co(f)(z_1) \ge \frac{||x-z_1||}{||y-z_1||}\co(f)(y)\\
\co(f)(y) &\ge \frac{||y-z_2||}{||x-z_2||}\co(f)(x)+\frac{||x-y||}{||x-z_2||} \co(f)(z_2) \ge \frac{||y-z_2||}{||x-z_2||}\co(f)(x),
\end{align*}
where $||\cdot ||$ denotes the Euclidean norm, so $\co(f)(y)\to \co(f)(x)$ as $y \to x$.
Next, instead suppose that $x\in \partial C$. Then take any linear function $L$ which is $0$ at $x$ and positive with $\inf_{\text{supp}(f)} L >0$, which exists as $f(x)$ is supported on a compact subset of the interior of $C$. We may further assume that $L|_{\text{supp}(f)}>||f||_{\infty}$ by replacing $L$ with $(1+||f||_{\infty})(\inf_{\text{supp}(f)} L) ^{-1}L$. Then $\co(f)$ is sandwiched between the constant function $0$ and the continuous function $L$ which agree at $x$, which implies that $\co(f)(x)=0$ and $\co(f)$ is continuous at $x$.

%First note that $\co(f)$ is continuous. Indeed, because it is concave we already know it is in fact Lipschitz continuous on $C^{\circ}$, so it suffices to check continuity on $\partial C$. We claim that in fact $\co(f)(x)=0$ for $x\in \partial C$ and $\lim f(x_i)\to 0$ as $x_i \to x\in \partial C$. First, to show for $x\in \partial C$ that $\co(f)(x)=0$, it suffices to find a concave function which is $0$ at $x$ and at least as large as $f$ everywhere else. To do this, note that as $f=0$ on a neighborhood of $\partial C$, we may find a linear function $L$ which is $0$ on $x$ and positive on the set $f\ne 0$, which we can scale up to ensure that $L>||f||_{\infty}$ at points where $f \ne 0$. A similar argument works for points near the boundary, which verifies the whole claim.

In particular, because $\co(f)$ is continuous and concave, it is approximated in the supremum norm by concave piecewise-linear functions from above. Let $c$ be a concave piecewise-linear approximation to $\co(f)$ with $c\ge \co(f)$, and $||c-\co(f)||_{\infty}\le \epsilon$  for some fixed $\epsilon$. Let $x_1,\ldots,x_N\in C$ be a finite collection of points for which the graph of $c$ is the upper convex hull of the points $(x_i,c(x_i))$ (note that here we use the fact that the domain is a polytope).

We note that $$\co(f)(x)=\sup \{\lambda_1f(x_1)+\ldots+\lambda_\ell f(x_{\ell}):\ell \in \mathbb{N}, \lambda_1,\ldots,\lambda_\ell \in [0,1], \sum \lambda_i=1, \sum \lambda_i x_i=x\}.$$

Hence, there exists $M$, points $\{x_{i,j}\}_{1 \le i \le N,1 \le j \le M}$ and parameters $\lambda_{i,j}\in [0,1]$ with $\sum_{j=1}^M \lambda_{i,j}=1$,  $\sum_{j=1}^M  \lambda_{i,j}x_{i,j}=x_i$, and $$\co(f)(x_i)\le \sum_{j=1}^M \lambda_{i,j}f(x_{i,j})+\epsilon.$$

%Hence there exist points $y_{i,1,\ell},\ldots,y_{i,k,\ell}\in C$ and parameters $t_{i,j,\ell}\in [0,1]$ such that $\sum_j t_{i,j,\ell}=1$, $\sum_j t_{i,j,\ell} y_{i,j,\ell} \to x_i$ and $\lim_{\ell \to \infty} \sum_j t_{i,j,\ell} f(y_{i,j,\ell})=\co(f)(x_i)$. Indeed, this follows as $\overline{\co(A_f)}=A_{\co(f)}$.

%consider the hypograph $$A_f=\{(x,y)\in C\times \mathbb{R} : y \le f(x)\}.$$ Then $\overline{\co(A_f)}=A_{\co(f)}$. Because $(x_i,\co(f)(x_i))\in A_{\co(f)}$, we therefore can find a sequence of points $(z_{i,j},w_{i,j})\in \co(A_f)$ converging to $(x_i,\co(f)(x_i))$ as $j\to \infty$. Now, the convex hull of a set is the union of simplices with vertices in that set, so we can find $(y_{i,j,1},u_{i,j,1}),\ldots,(y_{i,j,k},u_{i,j,k})\in A_f$ and parameters $t_{i,1,\ell},\ldots,t_{i,k,\ell}\in [0,1]$ such that $\sum_j t_{i,j,\ell}=1$ and $\sum_j t_{i,j,\ell}y_{i,j,\ell}=z_{i,j}$ and $\sum_j t_{i,j,\ell} u_{i,j,\ell}=w_{i,j}$. Then we still clearly have convergence when we replace $u_{i,j,\ell}$ with $f(y_{i,j,\ell})$ as we cannot replace $\co(f)(x_i)$ with anything larger in $(x_i,\co(f)(x_i))$ and keep the point in $A_{\co(f)}$.

%By passing to a convergent subsequence, we may also assume that the limits $\lim_\ell y_{i,j,\ell}= y_{i,j}\in C$, $\lim_\ell t_{i,j,\ell}= t_{i,j}$ and $\lim_\ell f(y_{i,j,\ell})$ exist. Furthermore, for $x_i\in \partial C$, as $\co(f)(x_i)=f(x_i)=0$ we may further assume that $y_{i,j,\ell}=x_i$ for all $j,\ell$.
Let $$f_{\epsilon}(x)=\begin{cases}f(x)+2\epsilon& \text{if }x = x_{i,j}\text{ for some $i,j$, and}\\ f(x)& \text{otherwise.}\end{cases}$$
We remark that $$\sum_{j=1}^M \lambda_{i,j}f_{\epsilon}(x_{i,j})=2\epsilon+ \sum_{j=1}^M \lambda_{i,j}f(x_{i,j}) \ge \co(f)(x_{i})+\epsilon \ge c(x_{i}).$$
Hence letting $g$ be the upper convex hull of the points $(x_{i,j},f_{\epsilon}(x_{i,j}))$, we have $g \ge c\ge f$. 

We claim that $\co(f_{\epsilon})(x)=g$. Indeed, we trivially have $g \le \co(f_{\epsilon})$, so it suffices to show $g \ge \co(f_{\epsilon})$.  For $x=x_{i,j}$ we clearly have $g(x) \ge f_{\epsilon}(x)$ and for $x\ne x_{i,j}$, we have $g(x)\ge f(x)=f_{\epsilon}(x)$. Hence $g \ge f_{\epsilon}$, so as $g$ is concave, $g \ge \co(f_{\epsilon})$.

Hence, $\co(f_{\epsilon})$ is the upper convex hull of finitely many points $(r_i,f_{\epsilon}(r_i))$. As $||f_{\epsilon}-f||_{\infty} \le 2\epsilon$ and $f_{\epsilon}\ge f$,  we have by our hypothesis,
\begin{align*}\int_C f^{\ast n}(x)-f(x)dx +2\epsilon|C| \ge \int_C (f_{\epsilon})^{\ast n}(x)-f_{\epsilon}(x) &\ge c_{k,n}\int_C \co(f_{\epsilon})(x)-f_{\epsilon}(x) dx \\&\ge c_{k,n}\int_C \co(f)(x)-f(x) dx - 2\epsilon c_{k,n}|C|,\end{align*}
where the first inequality follows from the fact that $f^{\ast n}+2\epsilon = (f+2\epsilon)^{\ast n} \ge (f_{\epsilon})^{\ast n}$. Letting $\epsilon \to 0$ we conclude.

\end{proof}

\begin{clm}
Suppose that \Cref{first_main} (resp. \Cref{constantconj}, \Cref{AsympTheorem}, \Cref{second_main}) is true when the domain $C$ is a simplex $T$, $f=0$ at the vertices of $T$ and $f\le 0$. Then it is true when the domain $C$ is a polytope $P$ and $\co(f)$ is the upper convex hull of finitely many points $(r_i,f(r_i))$.
\end{clm}
\begin{proof}
We prove this claim for \Cref{constantconj}, the other cases are similar.
Let $f$ be defined on a polytopal domain $C=P$ with $\co(f)$ the upper convex hull of finitely many points $(r_i,f(r_i))$. The domains of linearity of $\co(f)$ decompose $C$ into convex polytopes with vertices a subset of the $r_i$. Further subdivide this decomposition into triangulation $\mathcal{T}$. Then $\co(f|_TT)=\co(f)|_T$ for all $T\in \mathcal{T}$, so
$\int_C \co(f)(x)-f(x) dx = \sum_{T\in \mathcal{T}}\int_T\co(f|_T)(x)-f|_T(x) dx$ and
$$\int_C f^{\ast n}(x)-f(x) dx \ge \sum_T \int_T(f|_T)^{\ast n}(x)-f|_T(x) dx.$$
Hence it suffices to prove for every $T\in \mathcal{T}$ that
$$\int_T (f|_T)^{\ast n}(x)-f|_T(x) dx \ge c_{k,n}\int \co(f|_T)(x)-f|_T(x) dx.$$
As $\co(f|_T)$ is linear, and the inequality is preserved by subtracting linear functions from $f$, we may subtract $\co(f|_T)$ from $f$, after which $f=0$ at the vertices of $T$ and $f \le 0$ on $T$. Thus by hypothesis we are done.
\end{proof}

The above sequence of reductions gives the desired conclusion.

\end{proof}

\section{Hypersimplex Covering}
\label{HypersimplexSection}
We take $T$ to be the convex hull of the standard basis vectors $e_1,\ldots,e_{k+1}\in \mathbb{R}^{k+1}$. Recall that the $m$'th $k$-dimensional  hypersimplex for $1 \le m \le k$ is defined to be the region in $\mathbb{R}^{k+1}$ given by
$$P_{k,m}:=\left\{(x_1,\ldots,x_{k+1})\in [0,1]^{k+1}: \sum x_i=m\right\}.$$

\begin{defn}\label{def_partition}
Let $$\mathcal{B}_{k,\ell}=\left\{(x_1,\ldots,x_{k+1})\in \mathbb{Z}_{\ge 0}^{k+1} : \sum x_i=\ell\right\}.$$
\end{defn}
%It is classical that $|P_{k,m}|=A(k,m-1)$ where $A(k,\ell)$ is the Eulerian number counting the number of permutations of $S_k$ with exactly $\ell$ descents.
\begin{prop}
\label{subdivprop}
For $k,n\ge 1$ we have a polytopal subdivision
$$T=\bigcup_{m=1}^{\min(k,n)}\bigcup_{v \in \mathcal{B}_{k,n-m}}\frac{1}{n}P_{k,m}+\frac{1}{n}v.$$
\end{prop}
\begin{proof}
Note that because $\bigcup_{v \in \mathbb{Z}^{k+1}} [0,1]^{k+1}+v$ subdivide $\mathbb{R}^{k+1}$, the intersections $\bigcup_{v \in \mathbb{Z}^{k+1}}([0,1]^{k+1}+v)\cap nT$ form a polytopal subvidision of $nT$. Let $\mathcal{B}$ be the set of such $v$ for which $([0,1]^{k+1}+v)\cap nT$ is $k$-dimensional, such that $\bigcup_{v \in \mathcal{B}}([0,1]^{k+1}+v)\cap nT$ also forms a polytopal subvidision of $nT$.

Let $v\in \mathcal{B}$, and set $m=n-\sum v_i$. We first claim that $1 \le m \le k$ and
$$([0,1]^{k+1}+v)\cap \left\{\sum x_i=n\right\}=P_{k,m}+v.$$ Indeed, as $nT$ lies in the $\sum x_i=n$ hyperplane and $([0,1]^{k+1}+v)\cap nT$ is $k$-dimensional, we must have $n-k \le \sum v_i \le n-1$, i.e. $1 \le m \le k$. Then it is easy to see that $([0,1]^{k+1}+v)\cap \left\{\sum x_i=n\right\}=P_{k,m}+v$ by definition.

Next, we claim that $v\in \mathcal{B}_{k,n-m}$, i.e. that $v$ has no negative coordinates. Indeed, suppose that $v_1\leq -1$. Then $([0,1]^{k+1}+v) \cap nT \subset \{x_1=0\}\cap nT$ which is at most $k-1$-dimensional.

Finally, we claim that 
$$([0,1]^{k+1}+v)\cap \left\{\sum x_i=n\right\} \subset nT.$$ Indeed, as $v\in \mathcal{B}_{k,n-m}$, all coordinates are non-negative.

Conversely, if $v\in \mathcal{B}_{k,n-m}$, then $\frac{1}{n}P_{k,m}+\frac{1}{n}v=([0,1]^{k+1}+v)\cap \{\sum x_i=n\}=([0,1]^{k+1}+v)\cap nT$, where the last equality is because all coordinates are non-negative.

Therefore we have the polytopal subdivision $nT=\bigcup_{m=1}^{k}\bigcup_{v\in \mathcal{B}_{k,n-m}} P_{k,m}+v$, and as $\mathcal{B}_{k,n-m}=\emptyset$ for $m>n$, we conclude.
\end{proof}
\begin{prop}
For $1 \le m \le \min(k,n)$ and $v\in \mathcal{B}_{k,n-m}$, the points in $\frac{1}{n}P_{k,m}^{\circ}+\frac{1}{n}v$ (the interior of $\frac{1}{n}P_{k,m}+\frac{1}{n}v$) can be written as $\frac{x_1+\ldots+x_n}{n}$ with $n-m$ of the $x_i$ being vertices of $T$, but not with at least $n-m+1$ of the $x_i$ being vertices of $T$.
\end{prop}
\begin{proof}
For $y=\frac{1}{n}w+\frac{1}{n}v\in \frac{1}{n}P^{\circ}_{k,m}+\frac{1}{n}v$, we can write
$y=\frac{m \cdot (\frac{1}{m}w)+\sum v_i e_i}{n}$, and $\frac{1}{m}w\in \frac{1}{m}P_{k,m}^{\circ}\subset T$.

Conversely, suppose that we can write $y=\frac{x_1+\ldots+x_n}{n}$ with $x_1,\ldots,x_{n-m+1}$ vertices of $T$ and $x_{n-m+2},\ldots,x_n\in T$. Then $\lfloor ny \rfloor = v$, so we obtain the contradiction $$n-m=\sum_{i=1}^{k+1} \lfloor ny_i \rfloor\ge \sum_{i=1}^{k+1}\sum_{j=1}^n \lfloor (x_{j})_i\rfloor \ge \sum_{i=1}^{k+1}\sum_{j=1}^{n-m+1} \lfloor (x_{j})_i\rfloor = n-m+1.$$
\end{proof}

\section{\texorpdfstring{$m$}{m}-averageable sets}We now define a new notion of ``$m$-averageable'' subset of a simplex $T$.
\begin{defn}
Given a simplex $T$, say that a subset $S\subset T$ is ``$m$-averageable'' if there are mappings $H_1,\ldots,H_m:T \to T$ which are generically bijective of Jacobian $1$ such that $\frac{H_1+\ldots+H_m}{m}$ is a generically bijective map $T\to S$ with constant Jacobian $\frac{|S|}{|T|}$.
\end{defn}
The key property of an $m$-averageable set $S\subset T$ is the observation that
$$\int_S f_1\ast \ldots \ast f_m(x) dx \ge \frac{|S|}{|T|}\frac{\sum_{i=1}^m \int_T f_i(x) dx}{m}.$$ This observation will be used directly, and in a slightly modified form, in the propositions below.

\begin{exmp}
In three dimensions, the subsimplex $S\subset T$ defined by a vertex of $T$ and the opposite medial triangle is $2$-averageable.

Indeed, we can take $H_1$ to be the identity map and $H_2$ to be the linear map fixing the common vertex $v$ of $S$ and $T$ and cycling the remaining vertices $v_1\mapsto v_2\mapsto v_3 \mapsto v_1$. Then $\frac{H_1+H_2}{2}$ is a linear map sending the vertices of $T$ to the vertices of $S$, and is thus a constant Jacobian $\frac{1}{4}$ map $T\to S$.
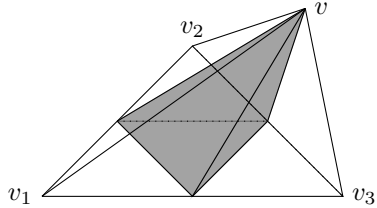
\begin{center}
    \begin{tikzpicture}[scale=0.5]
    \draw (0,0) node[anchor=east] {$v_1$}--(4,4) node[anchor=south] {$v_2$}--(8,0) node[anchor=west] {$v_3$}--(0,0);
    \filldraw[opacity=0.2] (2,2)--(6,2)--(4,0);
    \filldraw[opacity=0.2] (2,2)--(6,2)--(7,5);
    \filldraw[opacity=0.2] (6,2)--(4,0)--(7,5);
    \filldraw[opacity=0.2] (4,0)--(2,2)--(7,5);
    \draw (7,5)node[anchor=west] {$v$}--(0,0);
    \draw (7,5)--(4,4);
    \draw (7,5)--(8,0);
    \draw (7,5)--(2,2);
    \draw (7,5)--(6,2);
    \draw (7,5)--(4,0);
    \draw[dotted] (2,2)--(6,2);
    \draw[dotted] (4,4)--(8,0);
    \draw (2,2)--(4,0)--(6,2);
    \end{tikzpicture}
\end{center}
\end{exmp}

Recall that for $m\in \{1,\ldots,k\}$, we denote by  $P_{k,m}=[0,1]^{k+1}\cap \{\sum x_i=m\}$ for the $m$'th $k$-dimensional hypersimplex . The following two propositions reduce the theorems from the introduction to showing that certain hypersimplices embedded in $T$ are $m$-averageable.

\begin{prop}\label{fgprop}
Let $k\ge 1$ and $T$ the convex hull of the standard basis vectors in $\mathbb{R}^{k+1}$, and suppose that $\frac{1}{2}P_{k,2}$ is 2-averageable. Then for 
bounded functions $f,g:T\to \mathbb{R}$ with $f\leq 0$ and $f(x_i)=0$ for the vertices $x_i$ of $T$, we have
$$\int_T f\ast g(x)- \frac{f(x)+g(x)}{2} dx \ge \frac{k+1}{2^{k+1}}\int_T\co(f)(x)-f(x) dx.$$
\end{prop}

\begin{prop}\label{repeatedprop}
Let $T$ be the convex hull of the standard basis vectors in $\mathbb{R}^{k+1}$, and suppose that $\frac{1}{m}P_{k,m}$ is m-averageable for $m \le \min(k,n)$. Then for a
bounded function $f:T\to \mathbb{R}_{\leq 0}$ with  $f(x_i)=0$ for the vertices $x_i$ of $T$, we have
$$\int_T f^{\ast n}(x)-f(x) dx \ge c_{k,n}\int_T \co(f)(x)-f(x) dx$$
where $c_{k,n}$ is as in \Cref{constantconj}.
\end{prop}

\begin{proof}[Proof of \Cref{fgprop}]
Consider the $n=2$ polytope decomposition from \Cref{subdivprop}
$$T=\frac{1}{2}P_{k,2}\cup\bigcup_{i=1}^{k+1}\frac{e_i+T}{2}.$$
Because by hypothesis $\frac{1}{2}P_{k,2}$ is $2$-averageable, there are functions $H_1,H_2$ such that $H_1,H_2:T\to T$ are generically bijective with Jacobian $1$, and $\frac{H_1+H_2}{2}:T\to S$ is generically bijective with Jacobian $\frac{|\frac{1}{2}P_{k,2}|}{|T|}=1-\frac{k+1}{2^k}$. Then the result follows by adding the inequality
\begin{align*}\int_{\frac{1}{2}P_{k,2}} (f\ast g)(x) dx =&\left(1-\frac{k+1}{2^{k}}\right)\int_T f\ast g\left(\frac{H_1(x)+H_2(x)}{2}\right)dx\\\ge& \left(1-\frac{k+1}{2^{k}}\right)\frac{\int_T f(H_1(x)) dx+\int_T g(H_2(x)) dx}{2}\\
=&\left(1-\frac{k+1}{2^{k}}\right)\frac{\int_T f(x) dx+\int_T g(x) dx}{2}\end{align*}
to the inequality
\begin{align*}
    \int_{\frac{e_i+T}{2}}(f \ast g)(x) dx = \frac{1}{2^k} \int_{T} f\ast g \left(\frac{e_i+x}{2}\right) dx &\ge \frac{1}{2^{k+1}}\int_{T} f(e_i)+g(x) dx=\frac{1}{2^{k+1}}\int_{T} g(x) dx
\end{align*}
for $i=1,\ldots,k+1$.
\end{proof}

\begin{proof}[Proof of \Cref{repeatedprop}]
By \Cref{subdivprop}, there is a polytope subdivision
$$T=\bigcup_{m=1}^{\min(k,n)}\bigcup_{x\in \mathcal{B}_{k,n-m}}\frac{1}{n}P_{k,m}+\frac{1}{n}x$$
where $$\mathcal{B}_{k,\ell}=\left\{x=(x_1,\ldots,x_{k+1})\in \mathbb{Z}_{\ge 0}^{k+1}: \sum x_i=\ell\right\}.$$
Let $H^1_{k,m},\ldots,H^m_{k,m}$ be the functions associated to the $k$-averageable set $\frac{1}{m}P_{k,m}$. Then
\begin{align*}\int_{\frac{1}{n}P_{k,m}+\frac{1}{n}x}f^{\ast n}(x)dx&=\frac{1}{n^k}\cdot\frac{|P_{k,m}|}{|T|}\int_{T}f^{\ast n}\left(\frac{H_{k,m}^1(x)+\ldots+H_{k,m}^m(x)+x_1 e_1+\ldots+ x_{k+1}e_{k+1}}{n}\right)dx\\
&\ge \frac{1}{n^k}\cdot\frac{|P_{k,m}|}{|T|}\int_T\frac{f(H^1_{k,m}(x))+\ldots+f(H^m_{k,m}(x))+x_1f(e_1)+\ldots+x_{k+1}f(e_{k+1})}{n} dx\\
&= A(k,m-1)\cdot \frac{m}{n^{k+1}}\int_T f(x) dx.\end{align*}
Recalling that $\co(f)= 0$, summing these inequalities and using the Worpitzky identity that $\sum_m |\mathcal{B}_{k,n-m}|A(k,m-1)=\sum_m \binom{n+k-m}{k}A(k,m-1)=n^k$ yields the desired result.
\end{proof}

\section{Proofs of \texorpdfstring{\Cref{first_main}}{Theorem 1.1} and \texorpdfstring{\Cref{second_main}}{Theorem 1.5}}
By the propositions in the previous section, it will suffice to show that
\begin{align*}P_{1,1}\subset \mathbb{R}^2\\
P_{2,1},\frac{1}{2}P_{2,2}\subset \mathbb{R}^3\\P_{3,1},\frac{1}{2}P_{3,2},\frac{1}{3}P_{3,3}\subset \mathbb{R}^4
\end{align*}
are all $m$-averageable in the corresponding convex hull of standard basis vectors $T$ for $m=1$, $m=1,2$, and $m=1,2,3$ respectively.

The following lemma handles all cases except for $\frac{1}{2}P_{3,2}$.
\begin{lem}
$P_{k,1}$ is $1$-averageable and $\frac{1}{k}P_{k,k}$ is $k$-averageable for all $k\ge 1$.
\end{lem}
\begin{proof}
For $n=1$, $P_{k,1}=T$ so we may take $H$ to be the identity map and there is nothing to prove.

For $n=k$, let $\sigma$ be the linear map taking $e_1\mapsto e_2\mapsto \ldots \mapsto e_{k+1}\mapsto e_1$. Then $\sigma$ is an isometry, and so $H_i=\sigma^{i}$ is also an isometry. The average
$$\frac{H_1+\ldots+H_{k}}{k}$$ is the linear map taking $e_i \mapsto \frac{1}{k}\sum_{j \ne i}e_j$, which is a linear bijection from the simplex $T$ to the simplex $\frac{1}{k}P_{k,k}$.
\end{proof}

The following lemma therefore completes the proofs of \Cref{first_main} and \Cref{second_main}.
\begin{lem}
$\frac{1}{2}P_{3,2}$ is $2$-averageable.
\end{lem}
\begin{proof}
Decompose
$T=R_{12}\cup R_{23}\cup R_{34}\cup R_{41}$ where $R_{i(i+1)}$ is the simplex $$R_{i(i+1)}=\co\left(e_i,e_{i+1},\frac{e_1+e_3}{2},\frac{e_2+e_4}{2}\right).$$ Indeed, viewing $T$ as the $1$-dimensional cycle connecting $e_1 \to e_2 \to e_3 \to e_4 \to e_1$ coned off at the points $\frac{e_1+e_3}{2}$ and $\frac{e_2+e_4}{2}$, $R_{i(i+1)}$ corresponds to the line segment connecting $e_i\to e_{i+1}$ coned off at the points $\frac{e_1+e_3}{2}$ and $\frac{e_2+e_4}{2}$.

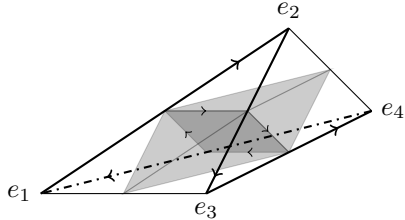
\begin{center}
\begin{tikzpicture}[scale=1.1]
\tikzset{->-/.style={decoration={
  markings,
  mark=at position #1 with {\arrow{>}}},postaction={decorate}}}
\draw (0,0) node[anchor=east]{$e_1$};
\draw (2,0) node[anchor=north]{$e_3$};
\draw (4,1)node[anchor=west]{$e_4$};
\draw (0,0)--(2,0);
\draw[thick, ->-=0.8]
(2,0)--(4,1);
\draw[thick, ->-=0.8, dashdotted] (4,1)--(0,0);
\draw[thick, ->-=0.8](0,0)-- (3,2)node[anchor=south]{$e_2$};
\draw [thick, ->-=0.9] (3,2)--(2,0);
\draw (3,2)--(4,1);

\filldraw[opacity=0.2] (1.5,1)--(2.5,1)--(3.5,1.5)--cycle;
\filldraw[opacity=0.2] (2.5,1)--(3,0.5)--(3.5,1.5)--cycle;
\filldraw[opacity=0.2] (1.5,1)--(2.5,1)--(3,0.5)--(2,0.5)--cycle;
\filldraw[opacity=0.2] (1.5,1)--(2.5,1)--(1,0)--cycle;
\filldraw[opacity=0.2] (2.5,1)--(3,0.5)--(1,0)--cycle;
\draw[draw=none,->-=0.5] (1.5,1)--(2.5,1);
\draw[draw=none,->-=0.5] (2.5,1)--(3,0.5);
\draw[draw=none,->-=0.5] (3,0.5)--(2,0.5);
\draw[draw=none,->-=0.5] (2,0.5)--(1.5,1);
%\draw (1.5,1)--(2.5,1)--(3,0.5)--(2,0.5)--cycle;
%\draw (1,0)--(1.5,1)--(3.5,1.5);
%\draw (1,0)--(2.5,1)--(3.5,1.5);
%\draw (1,0)--(3,0.5)--(3.5,1.5);
%\draw[dotted] (1,0)--(2,0.5)--(3.5,1.5);
\end{tikzpicture}
\end{center}
Let $H_1$ be the identity map and $H_2:T\to T$ be the piecewise linear local isometry defined by taking $R_{i(i+1)}\mapsto R_{(i+1)(i+2)}$, sending the vertices $e_i,e_{i+1},\frac{e_1+e_3}{2},\frac{e_2+e_4}{2}$ to $e_{i+1},e_{i+2},\frac{e_1+e_3}{2},\frac{e_2+e_4}{2}$, respectively.

Then $\frac{H_1+H_2}{2}$ takes $R_{i(i+1)}$ to $$S_{i(i+1)(i+2)}=\co\left(\frac{e_i+e_{i+1}}{2},\frac{e_{i+1}+e_{i+2}}{2},\frac{e_1+e_3}{2},\frac{e_2+e_4}{2}\right),$$
and the simplices $S_{123},S_{234},S_{341},S_{412}$ subdivide $\frac{1}{2}P_{3,2}$. Indeed, the octahedron $\frac{1}{2}P_{3,2}$ can be described as the one-dimensional cycle around the boundary of the square $\frac{e_1+e_2}{2}\to \frac{e_2+e_3}{2}\to \frac{e_3+e_4}{2}\to \frac{e_4+e_1}{2}\to \frac{e_1+e_2}{2}$ coned off at the points $\frac{e_1+e_3}{2}$ and $\frac{e_2+e_4}{2}$, and $S_{i(i+1)(i+2)}$ is the segment connecting $\frac{e_i+e_{i+1}}{2}$ and $\frac{e_{i+1}+e_{i+2}}{2}$ coned off at the points $\frac{e_1+e_3}{2}$ and $\frac{e_2+e_4}{2}$.

Hence $\frac{H_1+H_2}{2}$ is a bijection, and by symmetry has almost everywhere constant Jacobian. This shows $\frac{1}{2}P_{3,2}$ is $2$-averageable as desired.
\end{proof}

\section{Asymptotics for \texorpdfstring{$c_{n,k}$}{cnk} for \texorpdfstring{$k$}{k} fixed and \texorpdfstring{$n$}{n} large}
In this section we prove \Cref{AsympTheorem} that for $n \geq k+1$ we have $$c_{k,n} \ge 1-\binom{n}{k}\frac{k^{k+1}}{n^{k+1}}.$$

\begin{proof}[Proof of \Cref{AsympTheorem}]
Indeed, it suffices to show this $c_{k,n}$ works for functions $f$ on a simplex $C=T$ with $f=0$ at the vertices and $f\le 0$ everywhere by \Cref{reduction}. Set $T$ to be the convex hull of the standard basis vectors $e_1,\ldots,e_{k+1}$ in $\mathbb{R}^{k+1}$.

First, using the notation from \Cref{def_partition}, we claim that we have a covering
$$T=\bigcup_{v\in \mathcal{B}_{k,n-k}} \frac{kT+v}{n}.$$
Indeed, take $y\in T$, and consider $ny$. We can write $ny=w_1+\lfloor ny \rfloor$, and $\sum (\lfloor ny \rfloor)_i \ge n-k$. Write $\lfloor ny \rfloor= v+w_2$ with $v,w_2$ non-negative integral vectors such that $\sum v_i=n-k$. Then 
$$y=\frac{(w_1+w_2)+v}{n},$$ with $w_1+w_2\in kT$ and $v\in \mathcal{B}_{k,n-k}$.

We can then write
\begin{align*}\int_T f^{\ast n} &\ge \sum_{v\in\mathcal{B}_{k,n-k}}\int_{\frac{kT+v}{n}}f^{\ast n}(x) dx\\
&= \sum_{v\in \mathcal{B}_{k,n-k}}\left(\frac{k}{n}\right)^k\int_T f^{\ast n}\left(\frac{kx+v_1e_1+\ldots+v_ne_n}{n}\right) dx\\
&\ge \sum_{v\in \mathcal{B}_{k,n-k}}\left(\frac{k}{n}\right)^k\int_T \frac{kf(x)+v_1f(e_1)+\ldots+v_{k+1}f(e_{k+1})}{n}dx\\
&=\binom{n}{k}\frac{k^{k+1}}{n^{k+1}}\int_T f(x) dx.\end{align*}
As $\co(f)=0$ we can rearrange this to
$$\int_T f^{\ast n}(x)-f(x) dx \ge \left(1-\binom{n}{k}\frac{k^{k+1}}{n^{k+1}}\right)\int_T \co(f)(x)-f(x) dx.$$
\end{proof}

\appendix
\section*{Appendix}
\section{Non-sharp \texorpdfstring{$c_{k,n}$}{ckn} for \texorpdfstring{$f^{\ast n}$}{fn} for all \texorpdfstring{$k,n$}{k,n}}\label{Nonsharp}
We now discuss the existence of a non-sharp constant $c_{k,n}>0$ in all dimensions, i.e. that for all compact convex $C\subset \mathbb{R}^k$ and bounded measurable $f:C\to \mathbb{R}$, we have
$$\int f^{\ast n}(x)-f(x) \ge c_{k,n}\int_C \co(f)(x)-f(x) dx.$$
We can immediately deduce the existence of such constants from following result on the stability of Brunn-Minkowski for homothetic regions.

\begin{thm}[\cite{HomoBM}]
For any $k\in\mathbb{N}$ and $t\in (0,1)$, there are constants $c(k,t),d(k,t)>0$ such that for any $A\subset \mathbb{R}^{k+1}$ of positive measure if $|tA+(1-t)A|-|A|\le d(k,t)|A|$, then
$$|tA+(1-t)A|-|A|\ge c(k,t)|\co(A)\setminus A|,$$
where we write $\co(A)$ for the convex hull of $A$.
\end{thm}

Indeed, the existence of the constant $c_{k,n}$ for sup-convolution then follows by applying this theorem to the set $A=A_{f,-N}$ where $$A_{f,\lambda}=\{(x,y)\in C\times \mathbb{R}: \lambda \le y \le f(x)\},$$
$t=\frac{1}{n}$, and $-N\le \min(f)$ is sufficiently small so that the $d(k,\frac{1}{n})$ bound is satisfied (in fact this shows that we have a lower bound even if we restricted in the $n$-fold sup-convolution that $x_1=\ldots=x_{n-1}$).

For the benefit of the reader, we present here a simpler, more direct argument. Recall from \Cref{reduction} that it suffices to prove the theorem when we have the domain $C=T$ is a simplex, $f=0$ at the vertices of $T$, and $f\le 0$.

Call a translate of $T'\subset T$ of $\frac{1}{n^\ell}T$ ``good'' if there exists an absolute constant $C_{T'}$ (independent of $f$) such that
$$\int_{T'}f(x) dx \ge \frac{1}{n^{\ell(k+1)}}\int_T f(x) dx - C_{T'}\int_{T}f^{\ast n}(x)-f(x) dx.$$
We make the following observations
\begin{enumerate}
    \item $T$ is good.
    \item If $T'$ is good and $v$ is a vertex of $T$, then $\frac{(n-1)v+T'}{n}$ is good
    \item If $T',T''$ are good and of the same size, then $\frac{T'+(n-1)T''}{n}$ is good.
\end{enumerate}
The first observation is trivial. For the second, we note that
\begin{align*}\int_{\frac{(n-1)v+T'}{n}}f(x)dx &\ge \int_{\frac{(n-1)v+T'}{n}}f^{\ast n}(x)dx-\int_T f^{\ast n}(x)-f(x)dx\\&\ge \frac{1}{n^k}\int_{T'} \frac{(n-1)f(v)+f(x)}{n} dx - \int_T f^{\ast n}(x)-f(x)dx\\
&= \frac{1}{n^{k+1}}\int_{T'}f(x) dx -\int_T f^{\ast n}(x)-f(x) dx\\
&\ge \frac{1}{n^{(\ell+1)(k+1)}}\int_{T}f(x) dx - \left(1+\frac{C_{T'}}{n^{k+1}}\right)\int_T f^{\ast n}(x)-f(x) dx\end{align*}
For the third observation, we note that
\begin{align*}\int_{\frac{(n-1)T'+T''}{n}}f(x) dx &\ge \int_{\frac{(n-1)T'+T''}{n}}f^{\ast n}(x) dx-\int_T f^{\ast n}(x)-f(x) dx\\
&\ge \frac{(n-1)\int_{T'}f(x) dx+\int_{T''}f(x)dx}{n} - \int_T f^{\ast n}(x)-f(x) dx\\
& \ge \frac{1}{n^{\ell(k+1)}}\int_T f(x) dx - \left(1+ \frac{(n-1)C_{T'}+C_{T''}}{n}\right)\int_T f^{\ast n}(x)-f(x) dx.\end{align*}

If for some $\ell$ we have a family $\mathcal{A}$ of good translates of $\frac{1}{n^\ell}T$ which cover $T$, then adding the inequalities together, we obtain (recalling $\co(f)= 0$)
$$\left(\sum_{T'\in \mathcal{A}}C_{T'}\right)\int_T f^{\ast n}(x)-f(x) dx \ge \left(1-\frac{|\mathcal{A}|}{n^{\ell(k+1)}}\right)\int_T \co(f)-f(x) dx.$$
Hence if the total number of the simplices $|\mathcal{A}|$ is strictly less than $n^{\ell(k+1)}$, we are done.

From the second and third observations, for every face $F$ of $T$ (including $T$), the set of good translates of $\frac{1}{n^\ell}T$ is dense among the set of all translates of $\frac{1}{n^\ell}T$ incident to $F$. Together with the fact that simplices have a bounded inefficiency of covering space, we will be able to accomplish this task for a sufficiently large $\ell$. Indeed, as each simplex $\frac{1}{n^\ell}T$ covers a $\frac{1}{n^{\ell k}}$ volume of $T$, standard results from covering theory imply that we can find a family $\mathcal{A}$ with $|\mathcal{A}|=O(n^{\ell k})<n^\ell \cdot n^{\ell k}=n^{\ell(k+1)}$ for $\ell$ sufficiently large.

%%% AUTHOR: optional acknowledgments here
\section*{Acknowledgments} %%  you may comment this out if no Ackno
The authors would like to thank their advisor B\'ela Bollob\'as for his continuous support, and the anonymous reviewers for their helpful comments.

%%% AUTHOR:
%%% Bibliography goes here. Note that the arXiv cannot process bibtex
%%% or biber bibliographies.  Example of acceptable bibliograpy format:
\bibliographystyle{amsplain}

\begin{thebibliography}{vHST20b}

\bibitem[Ear16]{1611.06640}
Nick Early.
\newblock Combinatorics and representation theory for generalized permutohedra
  {I}: Simplicial plates, 2016.

\bibitem[EG69]{emerson1969asymptotic}
William~R Emerson and Frederick~P Greenleaf.
\newblock Asymptotic behavior of products ${C}^p= {C}+\dots+ {C}$ in locally
  compact abelian groups.
\newblock {\em Transactions of the American Mathematical Society},
  145:171--204, 1969.

\bibitem[FJ15]{FigJerJems}
Alessio Figalli and David Jerison.
\newblock Quantitative stability for sumsets in {$\mathbb{R}^n$}.
\newblock {\em J. Eur. Math. Soc. (JEMS)}, 17(5):1079--1106, 2015.

\bibitem[FJ19]{Semisum}
Alessio Figalli and David Jerison.
\newblock A sharp {F}reiman type estimate for semisums in two and three
  dimensional euclidean spaces.
\newblock {\em Ann. Sci. Ec. Norm. Supr.}, 2019.

\bibitem[FMMZ18]{fradelizi2018convexification}
Matthieu Fradelizi, Mokshay Madiman, Arnaud Marsiglietti, and Artem Zvavitch.
\newblock The convexification effect of {M}inkowski summation.
\newblock {\em EMS Surveys in Mathematical Sciences}, 5(1):1--64, 2018.

\bibitem[I74]{Ekeland}
Ekeland I.
\newblock Une estimation a priori en programmation non convexe.
\newblock {\em C.R. Acad. Sci., A; Fr.; Da. 1974; Vol. 279; No 4; pp. 149-151;
  Bibl. 5 Ref.}, 1974.

\bibitem[Ruz97]{ruzsa1997brunn}
Imre~Z Ruzsa.
\newblock The {B}runn--{M}inkowski inequality and nonconvex sets.
\newblock {\em Geometriae Dedicata}, 67(3):337--348, 1997.

\bibitem[Sta69]{starr1969quasi}
Ross~M Starr.
\newblock Quasi-equilibria in markets with non-convex preferences.
\newblock {\em Econometrica: journal of the Econometric Society}, pages 25--38,
  1969.

\bibitem[Str96]{Infconvsurvey}
Thomas Str\"{o}mberg.
\newblock The operation of infimal convolution.
\newblock {\em Dissertationes Math. (Rozprawy Mat.)}, 352:58, 1996.

\bibitem[vHST20a]{copro}
Peter van Hintum, Hunter Spink, and Marius Tiba.
\newblock Sets in $\mathbb{Z}^k$ with doubling $2^k+\delta$ are near convex
  progressions.
\newblock {\em arXiv preprint arXiv:2004.07264}, 2020.

\bibitem[vHST20b]{HomoBM}
Peter van Hintum, Hunter Spink, and Marius Tiba.
\newblock Sharp stability of {B}runn-{M}inkowski for homothetic regions.
\newblock {\em J. Eur. Math. Soc. (JEMS)}, (to appear) 2020+.

\end{thebibliography}

%% AUTHOR: You can generate such a bibliography from a .bib file by 
%% running pdflatex/bibtex/pdflatex/pdflatex and then pasting the .bbl file
%% between \begin{thebibliography} and \end{bibliography}

%%% AUTHOR: Include a short description of each author following the
%%% structure below. Use the same short tags used previously.  
%%% Use \imageat{} and \imagedot{} instead of "@" and "." in
%%% email addresses-this replaces the symbols with graphics to avoid 
%%% e-mail address harvesting from the .pdf file
\begin{dajauthors}
\begin{authorinfo}[pvh]
  Peter van Hintum\\
  Esm\'ee Fairburn Junior Research Fellow
  New College, Oxford University\\
  Oxford, UK\\
  peter\imagedot{}vanhintum\imageat{}new\imagedot{}ox\imagedot{}ac\imagedot{}uk \\
  \url{https://www.new.ox.ac.uk/peter-van-hintum}
\end{authorinfo}
\begin{authorinfo}[hs]
  Hunter Spink\\
  Szeg\"o Assistant Professor\\
  Stanford University\\
  Stanford, California, USA\\
  hspink\imageat{}stanford\imagedot{}edu \\
  \url{https://math.stanford.edu/~hspink/}
\end{authorinfo}
\begin{authorinfo}[mt]
  Marius Tiba\\
  Titchmarsh Research Fellow\\
  Oxford University\\
  Oxford, UK\\
  mt576\imageat{}cam\imagedot{}ac\imagedot{}uk
\end{authorinfo}
\end{dajauthors}

\end{document}